\documentclass[a4paper,12pt]{article}

\usepackage[top=2.5cm, bottom=2.5cm, left=2.5cm, right=2.5cm]{geometry}
 \usepackage{indentfirst}
\usepackage{graphicx}
\usepackage{enumerate}
\usepackage{amsthm}
\usepackage{amsfonts}
\usepackage{amsthm, amssymb, amsmath}

\usepackage{CJK}

\newtheorem{thm}{Theorem}[section]
\newtheorem{rem}[thm]{Remark}

\newtheorem{lem}[thm]{Lemma}

\numberwithin{equation}{section}

\title{\textbf{Generalization of Pontryagin Maximum Principle with Stochastic Initial Conditions} }
\author{{ Yuanzun Zhao} \\
{Department of Mathematics}\\ {Peking University,
Beijing 100871, China}\\
{\sf email: lukesweet@163.com} }
\date{}

\begin{document}

\maketitle
\renewcommand{\abstractname}{}
\begin {abstract}
\noindent
  {\bf Abstract}
{ }Based on Pontryagin Maximum Principle (PMP), this paper established a generalized PMP aiming at non-feedback control system with stochastic initial conditions. We proved the conclusion and show its coming back to PMP when the randomness collapses. Through the generalized PMP, a general method is introduced to solve expectation maximum problem of these systems and thereafter an example showed its feasibility.

\end{abstract}
\newpage

\section{{\textbf{Introduction}}}
Optimal control is a crucial component of modern control theory. Generally, once given initial conditions, cost function and admissible control set, we can use either variational method or maximum principle\cite{qy1} to find the optimal trajectory of certain control system without stochasticity. The optimal trajectory refers to the one that minimizes the cost function. To solve a given optimal control problem, one of the most used methods is to transform the system through Pontryagin Maximum Principle (PMP) into Hamilton equation\cite{qy2}. This method leads to predictable results in most cases with no randomness. However, stochasticity is not always avoidable in practice. This paper is focused on the control systems with stochastic initial conditions. Since a single preinstalled optimal control cannot deal with all existing cases, a new method is needed to substitute PMP.

We consider a generalization of PMP, which replacing the original optimal conditions with the expectation of cost function. In detail, we need a new conclusion named PMP* satisfying two requirements at least:
\begin{enumerate}[1)]
\item Given configuration space M and initial condition space $\Omega$, if the stochasticity satisfies certain condition, there exists a stochastic Lipschitzian curve $\lambda_t^*\in T_q^* M$ according to PMP*. Besides, the curve satisfies the maximum condition analogously.
\item When condition space $\Omega$ becomes a set of a single point, $\lambda_t^*$ converges to $\lambda_t$ given by PMP. That is, PMP* returns to PMP.
\end{enumerate}

When the generalized PMP is attained, generalized Hamilton equation can be constructed, and therefore optimal stochastic curve can be found. This paper will later give an example to show the feasibility of this method and if it comes back to PMP.

\vspace{2mm}
 \noindent {\bf Acknowledgement:} I would like to thank Mr. Zhangju Liu for his useful advice and great encouragement.
\vspace{1cm}

\section{{\textbf{Generalization of PMP from geometric viewpoint}}}
We now consider following optimal control problem from geometric viewpoint:
M is a configuration space, $U\subset \mathbb{R}^n$ is admissible control set. $q\in M$ is the state of system, and $u(t)\in U$ is a preinstalled control.

\begin{equation}
\dot{q}=f_u (q),      q\in M,    u\in U,
\end{equation}
with initial conditions
\begin{equation}
q(0)=q_0,
\end{equation}
Given cost function $¦Õ:M\times U\rightarrow  \mathbb{R}$, the problem is to find a control $u=\tilde{u}$ that minimizes $$J=\int_0^{t_1} \varphi(q(t),u(t)) dt.$$

We define Hamiltonian as $h_u(\lambda)=\langle \lambda,f_u (\lambda) \rangle$ and thereupon have the following geometric statement of PMP:
\theoremstyle{plain} \newtheorem{theorem}{Theorem}[section]
\begin{thm}[PMP]Given an admissible control $u=\tilde{u}(t)$, if (2.1),(2.2) have a solution $\tilde{q}(t)=q_{\tilde{u}(t)}(t)$, if\\
\begin{equation*}
\tilde{q}(t_1 )\in \partial A_{q_0}(t_1),
\end{equation*}
here $A_{q_0}(t_1)$ indicates the attainable set from point $q_0$ in $t_1$ time.\\
Then there exists a Lipschitzian curve $\lambda_t\in T_{\tilde{q}(t)}^* M,t\in [0,t_1]$ in the cotangent bundle, such that\\
\begin{eqnarray}
&&\lambda_t\not\equiv0,\\
&&\dot{\lambda}_t=\vec{h}_{\tilde{u}(t)}(\lambda_t),\\
&& h_{\tilde{u}(t)}(\lambda_t)=\max_{u\in U}\{h_{u(t)}(\lambda_t),
\end{eqnarray}
for almost all $t\in [0,,t_1]$
\end{thm}
The proof can be found in \cite{An1}.

Now we define right side of (2.2) as a random variable in probability space $(\Omega,F,P)$, rather than a fixed value in R. The probability density is hereinafter denoted as f(x).

The problem is to find the optimal control $\tilde{u}$ minimizing cost expectation $J=E\int_0^{t_1} \varphi(q(t),u(t)) dt$

Then we have the following conclusion£º
\begin{thm}\label{PMP2}[PMP*]Given an admissible control $u=\tilde{u}(t)$, if (2.1),(2.2) with a stochastic $q_0$ have a solution $\tilde{q}(t)=q_{\tilde{u}(t)}(t)$, if\\
\begin{equation*}
\tilde{q}(t_1 )\in \partial A_{q_0}^E (t_1),
\end{equation*}
here $A_{q_0}(t_1)$ indicates the expected attainable set from point $q_0$ in $t_1$ time. That is,\\
\begin{equation}
A_{x}^E (t)=\{\int_M a_xf(x)dx|a_x\in A_x (t)\}
\end{equation}

Meantime, u is injection: $\forall u, q_0^{(1)}\not= q_0^{(2)}, i.e. q_1^{(1)}\not= q_1^{(2)}$

Then there exists a stochastic Lipschitzian curve $\lambda_t^*\in T_{\tilde q(t)}^* M \times (\Omega,F,P),t\in [0,t_1]$, such that\\
\begin{eqnarray}
&&\lambda_t^*\not\equiv0,\\
&&\dot{\lambda}_t^*=\vec{h}_{\tilde{u}(t)}(\lambda_t^*),\\
&& Eh_{\tilde{u}(t)}(\lambda_t^*)=\max_{u\in U} Eh_{u(t)}(\lambda_t^*),
\end{eqnarray}
for almost all $t\in [0,t_1]$
\end{thm}
\begin{proof}
To prove the conclusion, a vector field depending on two parameters needs to introduced first:\\
\begin{equation*}
g_{\tau,u}={P_\tau^{t_1}}_*(f_{u(\tau)}-f_{\tilde u(\tau)}), \tau \in [0,t_1],u\in U
\end{equation*}

According to \cite{An2},
\begin{equation*}
q_u(t_1)=q_1\circ \vec{exp} \int_0^{t_1} g_\tau,u(\tau) dx
\end{equation*}

Thus we have the following lemmas:
\begin{lem}\label{Lem1}Let $\Gamma \in[0,t_1]$ be the set of Lebesgue points of the control $\tilde{u} (\cdot)$. If
\begin{equation*}
T_{q_1}M=cone\{g_{\tau,u} (q_1)|\tau \in T ,u\in U\}
\end{equation*}
Then
\begin{equation*}
q_1\in intA_{q_0}(t_1)=A_{q_0}(t_1)\backslash \partial A_{q_0}(t_1)
\end{equation*}
\end{lem}
The proof is in \cite{An3}
\begin{lem}\label{Lem2}Let $\Gamma \in[0,t_1]$ be the set of Lebesgue points of the control $\tilde{u} (\cdot)$. If
\begin{equation*}
T_{q_1}M=cone\{g_{\tau,u} (q_1)|\tau \in T ,u\in U\}, \forall q_1\in \Omega_1\triangleq\{q_1|q_0\in\Omega\}
\end{equation*}
Then
\begin{equation*}
Eq_1\in intA_{q_0}^E (t_1)=A_{q_0}^E (t_1)\backslash \partial A_{q_0}^E (t_1)
\end{equation*}
\end{lem}
\begin{proof}
For any $q_1^{q_0}\in \Omega_1$, Let $O_{q_1^{q_0}}\subset intA_{q_0} (t_1)$ be its neighbourhood. Since $intA_{q_0}$ is open for any $q_0$, the neighbourhoods are able to be attained. Consider the following set:\\
\begin{equation*}
O_{q_0}^E(t)=\{\int_M a_{q_0}f(x)dx|a_{q_0}\in O_{q_0}(t),q_0\in\Omega\}
\end{equation*}

By Lemma \ref{Lem1}, we have
\begin{equation*}
Eq_1=\int_M q_1f(x)dx\in O_{q_0}^E(t)\subset A_{q_0}^E(t_1)
\end{equation*}

Since $O_{q_0}^E(t)$ is open, there must be $Eq_1 \in int A_{q_0}^E(t_1)$.
\end{proof}

Now return to the proof of theorem \ref{PMP2}.
The terminal point of stochastic curve $\lambda_t^*$ satisfies $A_{x}^E (t)=\{int_M a_xf(x)dx|a_x\in A_x (t)\}$. By Lemma \ref{Lem2}, if this condition holds, $\forall q_1\in\Omega_1, a.e.$, origin point $o_{q_1}\in T_{q_1}M$ belongs to $\partial cone\{g_{\tau,u} (q_1)|\tau \in T ,u\in U\}$. Notice that $cone\{g_{\tau,u} (q_1)|\tau \in T ,u\in U\}$ is an open set.
thus,
\begin{equation*}
\exists \lambda_{t_1}^*\in T_{q_1}^*\times(\Omega,F,P), \lambda_{t_1}^*\not=0,
\end{equation*}
satisfies:
\begin{equation*}
E\langle \lambda_{t_1}^*,g_{t,u}(q_1)\rangle\leq 0, \forall t\in[0,t_1],u\in U a.e.
\end{equation*}
that is
\begin{equation*}
E\langle \lambda_{t_1}^*,P_{t*}^{t_1}f_u(q_1)\rangle\leq E\langle \lambda_{t_1}^*,P_{t*}^{t_1}f_{\tilde{u}}(q_1)\rangle
\end{equation*}
further,
\begin{equation*}
E\langle P_{t}^{t_1*}\lambda_{t_1}^*,f_u(q_1)\rangle\leq E\langle P_{t}^{t_1*}\lambda_{t_1}^*,f_u(q_1)\rangle
\end{equation*}

Then flow $P_t^{t_1}$ defines stochastic curve $\lambda_{t}^*$ in $ T_{q(t)}^*\times(\Omega,F,P)$:
\begin{equation*}
\lambda_{t}^*\triangleq P_t^{t_1*}\lambda_{t_1}^*\in T_{\tilde{q}(t)}^*\times(\Omega,F,P)
\end{equation*}

Injection condition ensures that the curve is well-defined.
In terms of this covector curve, the inequation above reads:
\begin{equation*}
E\langle \lambda_{t}^*,f_u(q_1)\rangle\leq E\langle \lambda_{t}^*,f_{\tilde{u}}(q_1)\rangle
\end{equation*}
which exactly equals to Hamilton maximum conditions:
\begin{equation*}
Eh_{\tilde{u}(t)}(\lambda_t^*)=\max_{u\in U} E\{h_{u(t)}(\lambda_t^*)
\end{equation*}

Besides, since the curve is fixed once the terminal point is fixed, the following equation holds:
\begin{equation*}
\forall q_1, \lambda_t=\lambda_{t_1}\circ (\vec{exp}\int_t^{t_1}f_{\tilde{u}(x)}dx)^*=\lambda_{t_1}\circ (\vec{exp}\int_t^{t_1}\vec{h}_{\tilde{u}(x)}dx)
\end{equation*}
i.e.
\begin{equation*}
\forall q_1, \dot{\lambda}_t=\vec{h}_{\tilde{u}(t)}(\lambda_t),
\end{equation*}
that is,
\begin{equation*}
\forall q_1, \dot{\lambda}_t^*=\vec{h}_{\tilde{u}(t)}(\lambda_t^*),
\end{equation*}

Thus, the existence of extremal stochastic curve is proved.
\end{proof}
\vspace{1cm}

\section{{\textbf{Application of generalized PMP in optimal control problems with stochastic initial conditions}}}
\subsection{{\textbf{Statement of problem}}}
Before applying PMP* into practice, we shall first make clear which problems are adaptive to it and which are not.

PMP is used to solve optimal control problems which have strict limitation for terminal point. However these limitations cannot be satisfied in stochastic control systems by a single preinstalled control. If vague limitation is required for terminal point, a penalty function $I(x),x\in M$ can be introduced to measure if a point is close enough to our expectation.

Thereout, we can define new cost function $\hat{J}(u)=E\int_0^{t_1}\varphi(q(t),u(t))dt+EI(q(t_1))$.
Let $I(x)$ be smooth, define
\begin{equation*}
\hat{\varphi}(q(t),u(t))={\varphi}(q(t),u(t))+(DI(q(t),{u}(t)))\cdot f_{u}(q(t)),
\end{equation*}
thus we get
\begin{equation}
\hat{J}=E\int_0^{t_1}\hat{\varphi}(q(t),u(t))dt.
\end{equation}

The problem has been transformed into a optimal control problem with stochastic initial conditions and no limitation to terminal conditions. The section is focused on problems of this kind.

Consider the following optimal control problem:
\begin{eqnarray}
&&\dot{q}=f_u(q), q\in M, u\in U\\
&&q(0)=q_0\in(\Omega,F,P)\\
&&t_1 fixed
\end{eqnarray}
and the cost function is defined as
\begin{equation}
\hat{J}=E\int_0^{t_1}\hat{\varphi}(q(t),u(t))dt,
\end{equation}

We extend this control system as follows:
\begin{equation*}
\hat{q}=\left(
\begin{array}{c}
J_{q_0}(u)\\
q
\end{array}
\right)
\end{equation*}
and extend the corresponding vector field:
\begin{equation*}
\hat{f}_u(q)=\left(
\begin{array}{c}
\varphi(q,u)\\
f_u(q)
\end{array}
\right)
\end{equation*}

Thus we get a new control system:
\begin{eqnarray}
&&\dot{\hat{q}}=\dot{f}_u(q), q\in M, u\in U\\
&&\dot{q}(0)=\dot{q}_0=\left(
\begin{array}{c}
0\\
q_0
\end{array}
\right)\\
&&t_1 fixed
\end{eqnarray}
where $\hat{q}_0$ is a (n+1)-dimension random vector.
\begin{rem}
Notice that if $\tilde{u}$ is optimal control, then the following condition holds:
\begin{equation*}
Eh_{\tilde{u}(t)}(\lambda_t^*)=\max_{u\in U} E\{h_{u(t)}(\lambda_t^*)
\end{equation*}
Then the terminal point satisfies:
\begin{equation*}
\tilde{\hat{q}}(t_1)\in \partial A_{\hat{q}_0}^E (t_1),
\end{equation*}
Therefore, PMP* can be applied to find the extremal curve.
\end{rem}
\vspace{0.5cm}

\subsection{{\textbf{Introduction of extremal parameter}}}
When use PMP* to solve optimal control problems, we need transform it into Hamilton equations through the Hamiltonian. However, directly using the previous Hamiltonian does not extinguish maximum from minimum of cost function. To make up the defect, we introduce new parameters $\nu$ and a new control $w$.
Define $y=J_{q_0}(u)$, and consider the following system:
\begin{eqnarray}
&&\dot y=\varphi(q,u)+w\\
&&\dot q= f_u(q)
\end{eqnarray}

Then the extremal stochastic curve in origin system corresponding to the control $w(t)\equiv0$. As a result, it comes to the boundary of attainable set at $t_1$. Apply PMP to it.\\
Define new Hamiltonian as follows:
\begin{equation}
\hat{h}_{(w,u)}(\nu,\lambda^*)=\langle \lambda^*,f_u\rangle +\nu(\varphi+w)
\end{equation}

The corresponding Hamilton system is
\begin{equation}
\left\{
\begin{aligned}
\frac{\partial \nu}{\partial t}=\frac{\partial E\hat h}{\partial y}=0\\
\frac{\partial Ey}{\partial t}=E\varphi+w\\
\dot\lambda_t^*=\vec h_{\tilde{u}(t)}(\lambda_t)
\end{aligned}
\right.
\end{equation}

The first equation stands for $v\equiv constant$.

In terms of this system, Hamilton Maximum condition is:
\begin{equation*}
E(\langle \lambda_t^*,f_{\tilde{u}(t)}\rangle+\nu\varphi(\tilde{q}(t),\tilde{u}(t)))=\max_{u\in U,w\geq0} E(\langle \lambda_t^*,f_{u(t)}\rangle+\nu\varphi(\tilde{q}(t),u(t)))+\nu w
\end{equation*}
Since the maximum of original system is attained, there must be $\nu\leq 0$, thus we can set $\nu=0$ in the right hand of maximum condition:
\begin{equation*}
E(\langle \lambda_t^*,f_{\tilde{u}(t)}\rangle+\nu\varphi(\tilde{q}(t),\tilde{u}(t)))=\max_{u\in U,w\geq0} E(\langle \lambda_t^*,f_{u(t)}\rangle+\nu\varphi(\tilde{q}(t),u(t)))
\end{equation*}
So we prove the following conclusion:
\begin{thm}If $u=\tilde u (t)$ is the optimal control for problem (3.9)-(3.11), that is, $\tilde u(t)$ minimizes $J(u)$. Define generalized Hamiltonian family: $h_u^\nu(\lambda)=\langle \lambda_t^*,f_{u(t)}\rangle+\nu\varphi(q(t),u(t))$, $\lambda_t\in T_q^* M \times (\Omega,F,P),t\in [0,t_1], q\in M, u\in U, \nu\in \mathbb{R}$£¬then there exists a stochastic Lipschitzian curve $\lambda_t^*\in T_{\tilde q(t)}^* M \times (\Omega,F,P),t\in [0,t_1]$, and a number $\nu\in \mathbb{R}$ such that\\
\begin{eqnarray}
&&(\lambda_t^*,\nu)\not\equiv 0,\\
&&\dot{\lambda}_t^*=\vec{h}_{\tilde{u}^\nu(t)}(\lambda_t^*),\\
&& Eh_{\tilde{u}(t)}^\nu(\lambda_t^*)=\max_{u\in U} Eh_{u(t)}^\nu(\lambda_t^*),\\
&&\nu\leq0,
\end{eqnarray}
for almost all $t\in [0,t_1]$
\end{thm}
\begin{rem}
since pair $(\lambda_t^*,\nu)$ can be multiplied by any positive number, only abnormal case $\nu=0$ and normal case $\nu=-1$ need to be considered. Besides, when solving maximum problems, $\nu\leq 0$ becomes $\nu\geq 0$, that is, analogously abnormal cases $\nu=0$ and normal case $\nu=1$.
\end{rem}
\vspace{0.5cm}

\subsection{{\textbf{Example}}}
To demonstrate this method more specifically, we consider its application to classic Cheapest Stop Problem.

The original problem can be described as follows: A train moves on the railway. We start braking the train at certain initial location and speed. The goal is to stop it with minimum expenditure of energy, which is assumed proportional to the integral of squared acceleration.

Its mathematic statements is
\begin{eqnarray*}
&&\ddot x=u, \\
&&x(0)=x_0\\
&&\dot x(0)=v_0\\
&&t_1 fixed, x(t_1)=x_1, \dot x(0)=0,
\end{eqnarray*}
find $u=\tilde u$ minimizing $J(u)=\int_0^{t_1}u^2dt$.

When the initial conditions become stochastic, we have the following description: At the time we start braking, its location and speed fits normal distribution independently. The goal is to stop it with as little energy as possible, simultaneously either close enough to expected destination and its speed low enough.

Its mathematic statement is
\begin{eqnarray*}
&&\ddot x=u, \\
&&x(0)\sim N(x_0,1)\\
&&\dot x(0)\sim N(v_0,1)\\
&&t_1 fixed,
\end{eqnarray*}
Cost function is defined as
\begin{equation*}
J(u)=\int_0^{t_1}u^2dt,
\end{equation*}
Penalty function is defined as
\begin{equation*}
I(u)=k(x^2(t_1)+\dot x^2(t_1)),
\end{equation*}
find $u=\tilde u$ minimizing $E(J(u)+I(u))$.
Reorganize the problem and we get the following new system:
\begin{eqnarray}
&&\left\{
\begin{array}{l}
\dot x_1=x_2\\
\dot x_2=u\\
\end{array}
\right.,x=\left(
\begin{array}{c}
x_1\\
x_2
\end{array}
\right)\in R^2, u\in R\\
&&x(0)=x^{(0)}\sim N_2(\left(
\begin{array}{c}
x_1\\
x_2
\end{array}
\right)\,\left(
  \begin{array}{cc}
  1 & 0\\
  0 & 1\\
  \end{array}
\right))\\
&&t_1 fixed,\\
&&\hat J(u)=E\int_0^{t_1}u^2-2k(x_1(t)+x_2(t))dt\rightarrow min.
\end{eqnarray}
The system satisfied injection condition, thus be applied with PMP*.
Its generalized Hamiltonian is
\begin{equation*}
h_u^\nu(\xi^*,x)=\xi^*_1 x_2+\xi^*_2 u+\nu(u^2-2k(x_1(t)+x_2(t))), \xi^*=\left(
\begin{array}{c}
\xi_1^*\\
\xi_2^*
\end{array}
\right)\in T_{x(t)}^*R^2,
\end{equation*}
First consider abnormal case $\nu=0$:
\begin{equation*}
h_u^0(\xi^*,x)=\xi^*_1 x_2+\xi^*_2 u.
\end{equation*}
If maximum of $h_u^0(\xi^*,x)$ exists, there must be $E\xi^*_2\equiv0$. This means $\tilde u \equiv0$, which contradicts nontrivial requirement.\\
Then consider the normal case $\nu=-1$:
\begin{equation*}
h_u^{-1}(\xi^*,x)=\xi^*_1 x_2+\xi^*_2 u-(u^2-2k(x_1(t)+x_2(t))).
\end{equation*}
Thus,
\begin{equation*}
\left\{
\begin{aligned}
\dot\xi_1^*=\frac{\partial\xi_1^*}{\partial t}=\frac{\partial h_u^{-1}}{\partial x_1}=0\\
\dot\xi_2^*=\frac{\partial\xi_2^*}{\partial t}=\frac{\partial h_u^{-1}}{\partial x_2}=\xi_1^*
\end{aligned}
\right.
\end{equation*}
Due to the expected Hamiltonian maximum generated by U,
\begin{equation*}
Eh_u^{-1}(\xi^*,x)=E(\xi^*_1 x_2+\xi^*_2 u-(u^2-2k(x_1(t)+x_2(t)))).
\end{equation*}
\begin{equation*}
\frac{\partial Eh_u^{-1}}{\partial u}=0
\end{equation*}
Thus we get
\begin{equation*}
\tilde u(t)=\frac{1}{2}E\xi_2^*(t)=\alpha t+\beta.
\end{equation*}

The optimal control shall be linear. We can easily find the optimal control when put the conclusion into origin system.

\vspace{1cm}

\section{{\textbf{Conclusion}}}
The paper is focused on non-feedback control system with stochastic initial conditions through generalizing Pontryagin Maximum Principle in certain degree. Further we establish a generalized method to solve problems of this kind. It needs to be explained that this method can be just applied to non-feedback systems, the cost expectation of which can be minimized. Through the example in Section 3, we found that procedure is analogous to the systems without stochasticity. The only difference is that much larger calculation is required after expectation is introduced.

Certainly, the method is with clear limitation. In section 2 we see that PMP* becomes unavailable either when the system is without injection conditions or when the penalty function is not smooth enough. In terms of the former, a possible solution is to further limit the control. In fact, we only need injection condition for the extremal curve corresponding to the optimal control. For the latter, smoothing the penalty function can be useful. Actually, the system can be divided into several smooth ones as far as the non-smooth point set is finite.

In addition, there is considerable potentiality in this subject. For instance, assume a control system can be disturbed instantaneously at certain points. These systems can be divided into several systems with stochastic initial conditions and thereupon each can be easily solved.
Another example is the systems with continuous disturbance. Systems of these kinds are yet to be researched in detail.

\end{document}